\documentclass[11pt, reqno]{amsart}
\usepackage[dvipsnames,usenames]{color}
\usepackage{amsmath}
\usepackage{amsfonts}
\usepackage{amssymb}
\usepackage{mathtools}
\usepackage{mathrsfs}
\usepackage{amsthm}
\usepackage{hyperref}
\usepackage{url}
\usepackage{comment}
\date{}

\newtheorem{lem}{Lemma}[section]

\newtheorem{con}{Conjecture}[section]

\theoremstyle{definition}

\numberwithin{equation}{section}

\theoremstyle{remark}

\numberwithin{equation}{section}

\makeatletter
\@namedef{subjclassname@2020}{\textup{2020} Mathematics Subject Classification}
\makeatother

\begin{document}
	\setcounter{page}{1}

	\title[Convex combinations of univalent harmonic
mappings]
{ Note on the Coefficient Conjecture of Clunie and Sheil-Small on the univalent harmonic mapping }

\author[O. Mishra, A. \c{C}etinkaya]{Omendra Mishra$^{*}$ and  Asena \c{C}etinkaya}

\address{\noindent $^*$Omendra Mishra\vskip.03in
	Department of Mathematical and Statistical Sciences,
	Institute of Natural Sciences and Humanities, Shri Ramswaroop Memorial University, Lucknow 225003, India.\\}

\email{mishraomendra$@$gmail.com}	
		\address{\noindent Asena \c{C}etinkaya  \vskip.03in
		Department of Mathematics  and Computer Science,
		Istanbul K\"{u}lt\"{u}r University, Istanbul, T\"{u}rkiye.}
		\email{asnfigen$@$hotmail.com}



	
\thanks{$^*$Corresponding author: Omendra Mishra,
	Email:mishraomendra@gmail.com}
	\date{\today}

	\begin{abstract}
In this article, we construct generalized harmonic univalent mappings and find its coefficients bounds. We present the counterexample to validate the coefficient conjecture proposed by Clunie and Sheil-Small for the class of functions $f=h+\overline{g}\in \mathcal{S}_{\mathcal{H}}$ with the help of these examples we improve the conjecture bounds of class $\mathcal{S}_{\mathcal{H}}$.
	\end{abstract}
\maketitle


\section{Introduction}
Let $\mathcal{H}$ represent the class of complex-valued functions $f=u+iv$ that are harmonic on the unit disk $\mathbb{D}=\left\{ z\in \mathbb{C}:\left\vert
{z}\right\vert <1\right\} ,$ where $u$ and $v$ are real-valued harmonic
functions in $\mathbb{D}$. A function $f$ ${\in \mathcal{H}}$ can be
expressed as $f=h+\overline{g},$ where $h$ and $g$ are analytic in $\mathbb{D%
},$ called the analytic and co-analytic parts of $f$, respectively. 

The
Jacobian of $f=h+\overline{g}$ is given by $$J_{f}(z)=|{h^{\prime }(z)}|^{2}-|%
{g^{\prime }(z)}|^{2}.$$
A classical result of Lewy \cite{lewy} showed that each harmonic function $f=h+\overline{g}%
\in \mathcal{H}$ is locally univalent and sense-preserving 
if and only if $J_{f}(z)>0$ in $\mathbb{D}$, which is equivalent to the
existence of an analytic function $\omega _{f}(z)=g^{\prime }(z)/h^{\prime
}(z)$ in $\mathbb{D}$ such that
$|{\omega }_{f}{(z)}|<1$. As usual, $\omega _{f}$ is called the dilatation of $f$. For more details, we  refer the reader to \cite{am,cm,MD,MNW,MR 12,du}.

The familiar class of univalent  harmonic functions $f=h+\overline{g}\in \mathcal{H}$ with normalized conditions $h(0)=g(0)=0$ and $h^{\prime }(0)=1$ is denoted by $\mathcal{S}_{\mathcal{H}}$. For a mapping $f=h+\overline{g}\in \mathcal{S}_{\mathcal{H}}$, $h$ and $g$ are of the forms
\begin{equation}
	h(z)=z+\sum_{n=2}^{\infty }a_{n}z^{n}\text{ \ and \ }%
	g(z)=\sum_{n=1}^{\infty }b_{n}z^{n}\ \ \left(\left\vert
	b_{1}\right\vert <1;z\in \mathbb{D}\right).  \label{HG}
\end{equation}

The subclass of functions $f=h+\overline{g}\in \mathcal{S}_{\mathcal{H}}$ that satisfy the additional condition $g^{\prime }(0)=0$ is denoted by $\mathcal{S}^0_{\mathcal{H}}$. Moreover, the subclasses of starlike, convex, and close-to-convex functions $f$ in $\mathcal{S}_{\mathcal{H}}$ $\left( \mathcal{S}^0_{\mathcal{H}} \right) $ are, respectively, denoted by $\mathcal{S}^\ast_{\mathcal{H}} $ $\left( \mathcal{S}^{\ast,0}_{\mathcal{H}}\right) $, $\mathcal{K}_{\mathcal{H}} $ $\left( \mathcal{K}^0_{\mathcal{H}}\right) $ and $\mathcal{C}_{\mathcal{H}} $ $\left( \mathcal{C}^{0}_{\mathcal{H}}\right)$.

A domain $\Omega $ $\subset\mathbb{C}$ is said to be convex in the direction $\varphi \in \lbrack 0,\pi ),$ if for all $t$ $\in $ $\mathbb{C}$, the set $\Omega $ $\cap $ $\{t+re^{i\varphi }:r\in \mathbb{R}\}$ is connected or empty. A function is said to be convex in the direction $\varphi $ if it maps $\mathbb{D}$ univalently to a domain  convex in the direction $\varphi$.

In 1984, Clunie-Sheil-Small \cite{JT 84} 
introduced a method, known as \textit{shearing technique}, to construct a univalent harmonic mapping from a related conformal mapping. By this method, a generalized result can be presented as follows:
\begin{lem}
	\label{lemma} A locally univalent harmonic function $f=h+\overline{g}$ in $\mathbb{D}$ is a univalent harmonic mapping of $\mathbb{D}$ onto a domain convex in the direction $\varphi $ if and only if $h-e^{2i\varphi }g$ is a univalent analytic mapping of $\mathbb{D}$ onto a domain convex in the direction $\varphi $.
\end{lem}

By shearing the \textit{Koebe function} $$k(z)=\frac{z}{(1-z)^2}$$ horizontally convex with dilatation $w(z) = z$,  Clunie-Sheil-Small \cite{JT 84} introduced the \textit{harmonic Koebe function} for $f=h+\overline{g}\in \mathcal{S}^{0}_{\mathcal{H}}$ 
\begin{equation}\label{equa13}
	k_0(z)=\frac{z-\frac{1}{2}z^2+\frac{1}{6}z^3}{(1-z)^3}+\overline{\frac{\frac{1}{2}z^2+\frac{1}{6}z^3}{(1-z)^3}}.
\end{equation}
They conjectured that the Taylor coefficients of the series of $h$ and $g$ satisfy the inequalities
$$|a_n|\leq \frac{1}{6}(2n+1)(n+1), \ \ |b_n|\leq \frac{1}{6}(2n-1)(n-1)$$
for all $n\geq2$. 

By shearing the conformal mapping $$s(z)=h(g)+g(z)=\frac{z}{1-z}$$ vertically convex with dilatation $w(z)=-z$, Clunie-Sheil-Small \cite{JT 84} also considered the harmonic half-plane mapping  
\begin{equation}\label{eq:fo}
	s_0(z)=\frac{z-\frac{1}{2}z^2}{(1-z)^2}-\overline{\frac{\frac{1}{2}z^2}{(1-z)^2}}.
\end{equation}
They proved that the
Taylor coefficients of the series of $h$ and $g$ satisfy the inequalities
$$|a_n|\leq \frac{n+1}{2},\  \ |b_n|\leq \frac{n-1}{2}$$
for all $n\geq2$.

Clunie-Sheil-Small \cite{JT 84} also conjectured that for all functions of class  $f=h+\overline{g}\in \mathcal{S}_{\mathcal{H}}$, satisfy the inequalities
$$|a_n|<\frac{1}{3}(2n^{2}+1), \ \ |b_n|< \frac{1}{3}(2n^{2}+1)$$
for all $n\geq2$.
 
In this paper, we construct generalized harmonic Koebe function for harmonic mappings $f_c=h_c+\overline{g}_c\in \mathcal{S}^{c}_{\mathcal{H}}$, where $c>0$.
We generalized the \textit{Koebe function} $$k_c(z)=h_c-g_c=\frac{2c}{(1+c)}\frac{z}{(1-z)^2}$$  with dilatation     $w(z) = ( (1-c)/(1+c))+z)/(1+((1-c)/(1+c))z)$, where $c>0$, and get the generalized harmonic Koebe function as
\begin{equation}\label{equa131}
	k_c(z)=\frac{z-\frac{c}{(1+c)}z^2+\frac{1}{3(1+c)}z^3}{(1-z)^3}+\overline{\frac{((1-c)/(1+c))z+\frac{c}{(1+c)}z^2+\frac{1}{3(1+c)}z^3}{(1-z)^3}}.
\end{equation}
Taylor coefficients of the series of $h$ and $g$ satisfy the inequalities
$$|a_n|\leq \frac{2n^2+3cn+1}{3(1+c)}, \ \ |b_n|\leq  \frac{2n^2-3cn+1}{3(1+c)}$$
for all $n\geq2$. For $c=1$, we get the harmonic Koebe function of the form \ref{equa13} (see also \cite{Meteljevic 2,Meteljevic 3,Meteljevic 4}). 

Generalized right-half plane mapping for the class of harmonic functions $f_c=h_c+\overline{g}_c\in \mathcal{S}_{\mathcal{H}}$, first studied by Muir \cite{Muir}. By shearing the conformal mapping $$s_c(z)=h_c+g_c=\frac{2}{(1+c)}\frac{z}{1-z}$$ vertically convex with dilatation $w(z)=((1-c)/(1+c))-z)/(1-((1-c)/(1+c))z)$, where $c>0$, and considered the generalized harmonic right-half plane mapping  
\begin{equation}\label{eq:fo1}
	s_c(z)=\frac{(1+c)z-z^{2}}{(1+c)(1-z)^{2}}+\overline{\frac{(1-c)z-z^{2}}{(1+c)(1-z)^{2}}}.
\end{equation}
Muir\cite{Muir} proved that the
Taylor coefficients of the series of $h$ and $g$ satisfy the inequalities
$$|a_n|\leq \frac{1+nc}{1+c},\  \ |b_n|\leq \frac{1-nc}{1+c}$$
for all $n\geq2$.

For $c=1$, we get the harmonic right-half plane mapping of the form \ref{eq:fo}.

 We also construct another generalized harmonic Koebe function for harmonic mappings $f_a=h_a+\overline{g}_a\in \mathcal{S}^{a}_{\mathcal{H}}$, where $a \in (-1,1)$.
 We generalized the \textit{Koebe function} $$k_a(z)=h_a-g_a=\frac{(1-a)z}{(1-z)^2}$$  with dilatation $w(z) = (a+z)/(1+az)$, where $-1<a<1$, and get the generalised harmonic Koebe function as
 \begin{equation}\label{equa131}
 k_a(z)=\frac{z+\frac{a-1}{2}z^2+\frac{1+a}{6}z^3}{(1-z)^3}+\overline{\frac{az+\frac{1-a}{2}z^2+\frac{1+a}{6}z^3}{(1-z)^3}}.
 \end{equation}
 Taylor coefficients of the series of $h$ and $g$ satisfy the inequalities
 $$|a_n|\leq \frac{2(1+a)n^2+3(1-a)n+(1+a)}{6}, \ \ |b_n|\leq \frac{2(1+a)n^2+3(a-1)n+(1+a)}{6}$$
 for all $n\geq2$. For $a=0$, we get the harmonic Koebe function of the form \ref{equa13}.
 
 Generalized right-half plane mapping for the class of harmonic functions $f_a=h_a+\overline{g}_a\in \mathcal{S}^{a}_{\mathcal{H}}$, first studied by Liu and Ponnusamy \cite{Zs}. By shearing the conformal mapping $$s_a(z)=h_a+g_a=\frac{(1+a)z}{1-z}$$ vertically convex with dilatation $w(z)=a-z/1-az$, where $-1<a<1$, they considered the harmonic right-half plane mapping  
 \begin{equation}\label{eq:fo1}
 s_a(z)=\frac{z-\frac{1}{2}(1+a)z^{2}}{(1-z)^{2}}+\overline{\frac{az-\frac{1}{2}(1+a)z^{2}}{(1-z)^{2}}}.
 \end{equation}
 They proved that the
 Taylor coefficients of the series of $h$ and $g$ satisfy the inequalities
 $$|a_n|\leq \frac{(1+a)+n(1-a)}{2},\  \ |b_n|\leq \frac{(1+a)-n(1-a)}{2}$$
 for all $n\geq2$.
 
 For $a=0$, we get the harmonic right-half plane mapping of the form \ref{eq:fo}. Mappings  defined by $f_c=h_c+\overline{g}_c\in \mathcal{S}^{c}_{\mathcal{H}}$, where $c>0$ easily obtained by sustituting $a=(1-c)/(1+c)$ in Mappings defined by $f_a=h_a+\overline{g}_a\in \mathcal{S}^{a}_{\mathcal{H}}$, where $a \in (-1,1)$ (see \cite{Mishra cetinkaya}).

\vskip .10in

\section{Coefficient Conjecture of univalent harmonic mappings}
In this section, we present the coefficient conjecture proposed by Clunie-Sheil-Small \cite{JT 84} and provide the counterexample to validate the conjecture of class of functions $f=h+\overline{g}\in \mathcal{S}_{\mathcal{H}}$ also provided the improved conjecture for class $\mathcal{S}_{\mathcal{H}}$.
\begin{con}\label{clunie}
	Functions $f=h+\overline{g}\in \mathcal{S}^0_{\mathcal{H}}$, satisfy the inequalities
	$$|a_n|\leq \frac{1}{6}(2n+1)(n+1), \ \ |b_n|\leq \frac{1}{6}(2n-1)(n-1)$$
	for all $n\geq2$.  	
\end{con}
\begin{proof}
Bounds of the conjecture attained by harmonic koebe function defined in \eqref{equa13}. 	
\end{proof}
\begin{con}\label{con2}
All functions $f=h+\overline{g}\in \mathcal{S}_{\mathcal{H}}$, satisfy the inequalities

\begin{equation}
|a_n|<\frac{1}{3}(2n^{2}+1), \ \ |b_n|< \frac{1}{3}(2n^{2}+1) \label{web}
\end{equation} 
for all $n\geq2$. 	
\end{con}
\begin{proof}
   	By generalized harmonic Koebe function \eqref{equa131}.  $$k_c(z)=h_c-g_c=\frac{2c}{(1+c)}\frac{z}{(1-z)^2}$$  with dilatation $w(z) = ((1-c)/(1+c))+z)/(1+((1-c)/(1+c))z)$, where $c>0$, and get the generalized harmonic Koebe function as
   	\begin{equation*}
   	k_c(z)=\frac{z-\frac{c}{(1+c)}z^2+\frac{1}{3(1+c)}z^3}{(1-z)^3}+\overline{\frac{((1-c)/(1+c))z+\frac{c}{(1+c)}z^2+\frac{1}{3(1+c)}z^3}{(1-z)^3}}.
   	\end{equation*}
   	Taylor coefficients of the series of $h$ and $g$ satisfy the inequalities
   	$$|a_n|\leq \frac{2n^2+3cn+1}{3(1+c)}, \ \ |b_n|\leq  \frac{2n^2-3cn+1}{3(1+c)}$$
   	for all $n\geq2$. Also for  generalized harmonic Koebe function for harmonic mappings $f_a=h_a+\overline{g}_a\in \mathcal{S}^{a}_{\mathcal{H}}$, where $a \in (-1,1)$.
   	We have $$k_a(z)=h_a-g_a=\frac{(1-a)z}{(1-z)^2}$$  with dilatation $w(z) = (a+z)/(1+az)$, where $-1<a<1$, and get the generalised harmonic Koebe function as 
   	\begin{equation}
   	k_a(z)=\frac{z+\frac{a-1}{2}z^2+\frac{1+a}{6}z^3}{(1-z)^3}+\overline{\frac{az+\frac{1-a}{2}z^2+\frac{1+a}{6}z^3}{(1-z)^3}}.
   	\end{equation}
   	Taylor coefficients of the series of $h$ and $g$ satisfy the inequalities
   	$$|a_n|\leq \frac{2(1+a)n^2+3(1-a)n+(1+a)}{6}, \ \ |b_n|\leq \frac{2(1+a)n^2+3(a-1)n+(1+a)}{6}$$
   	for all $n\geq2$. These bounds prove that inequalities in \eqref{web} is true for all class $\mathcal{S}_{\mathcal{H}}$, also these examples provided the equality for class $\mathcal{S}_{\mathcal{H}}$.
   	\end{proof}
   	We proposed the improved conjecture for all functions $f=h+\overline{g}\in \mathcal{S}_{\mathcal{H}}$ in the following form
   	\begin{con}
   		All functions $f=h+\overline{g}\in \mathcal{S}_{\mathcal{H}}$, satisfy the inequalities
   			\begin{equation}
   		|a_n|\leq \frac{2(1+a)n^2+3(1-a)n+(1+a)}{6}, \ \ |b_n|\leq \frac{2(1+a)n^2+3(a-1)n+(1+a)}{6}
   		\end{equation} 
   		for all $n\geq2$, where $-1<a<1$. 	
   	\end{con}
   \vskip .05in
   \noindent{\bf  Funding.} Not Applicable.
   
   \vskip .05in
   \noindent{\bf Conflicts of interest.} The authors declare that they have no conflict of interest.
   
   \vskip .05in
   \noindent{\bf Data availability statement.}  The datasets generated during and/or analysed during the current study are available from the corresponding author on reasonable request.


\begin{thebibliography}{99}
	\bibitem{am}
	A. Aleman and  A. Constantin, Harmonic maps and ideal fluid flows, \textit{Arch. Ration. Mech. Anal.} \textbf{204} (2012),
	479--513. 
	\vskip.05in

	\bibitem{JT 84} J. Clunie and T. Shiel-Small, Harmonic univalent functions, \textit{Ann. Acad. Sci. Fenn. Ser. A. I Math.} \textbf{9} (1984), 3--25.
	\vskip.05in
	\bibitem{cm}
	O. Constantin and M. J. Mart\'{i}n,  A harmonic maps approach to fluid flows, \textit{Math. Ann.} \textbf{369} (2017), 1--16.
	
	
	\vskip.05in
	\bibitem{MD} M. Dorff, Convolutions of planar harmonic convex mappings, \textit{Complex Var. Theory Appl.} \textbf{45} (2001), 263--271.
	\vskip.05in
	\bibitem{MNW} M. Dorff, M. Nowak\ and\ M. Wo\l oszkiewicz, Convolutions of harmonic convex mappings, \textit{Complex Var. Elliptic Equ.} \textbf{57} (2012), 489--503.
	\vskip.05in
	\bibitem{MR 12} M. Dorff and J. Rolf, \textit{Anamorphosis, mapping problems, and harmonic univalent functions}, Explorations in Complex Analysis, Math. Assoc. of America, Inc., Washington, DC (2012), 197--269.
	\vskip.05in
	\bibitem{du} P. Duren, \textit{Harmonic mappings in the plane}, Cambridge University Press, Cambridge, 2004.
	\vskip.05in
	\bibitem{lewy} H. Lewy, On the non-vanishing of the Jacobian in certain one-to-one mappings, \textit{Bull. Amer. Math. Soc.} \textbf{42} (1936), 689--692.
	\vskip.05in
	\bibitem{Zs} Z. Liu and S. Ponnusamy, Univalency of convolutions of
	univalent harmonic right half-plane mappings, Comput. Methods Funct. Theory
	\textbf{17} (2017), no.~2, 289--302.
	\vskip.05in

	\bibitem{Muir} S. Muir, Harmonic mappings convex in one or every direction, \textit{Comput. Methods Funct. Theory}  \textbf{6} (2012), 221--239.
	\vskip.05in
	
		\bibitem{Mishra cetinkaya} A. Cetinkaya, O. Mishra,  Radii results of Convex combinations and convolution  of generalized univalent harmonic
		mappings, \textit{Filomat}  \textbf{40}, no.~8, (2026).
		\vskip.05in
		
	
	\bibitem{Meteljevic} M. Mateljevic, The Lower Bound for the Modulus of the Derivatives and Jacobian of Harmonic Injective Mappings. Filomat 2015, \textbf{29}, 221--244.
	\vskip.05in
	\bibitem{Meteljevic 1} M.	Mateljevic, N. Mutavdzic, The Boundary Schwarz Lemma for Harmonic and Pluriharmonic Mappings and Some Generalizations. Bull. Malays. Math. Sci. Soc. 2022, \textbf{45}, 3177--3195.
	\vskip.05in
	\bibitem{Meteljevic 2}  M. Mateljevic, Quasiconformal and Quasiregular harmonic analogues of
	Koebe's Theorem and Applications, Ann. Acad. Sci. Fenn.-M, Vol 32, (2007),
	No 2, 301--315.
	\vskip.05in
	\bibitem{Meteljevic 3} M. Mateljevic, N. Mutavdzic and B. N. Ornek, Note on some classes of
	holomorphic functions related to Jack’s and Schwarz’s lemma, Appl. Anal.
	Discrete Math. 16 (2022), 111--131. https://doi.org/10.2298/AADM200319006M.
	\vskip.05in
	\bibitem{Meteljevic 4} M. Mateljevic, Versions of Koebe 1/4 theorem for analytic and
	quasiregular harmonic functions and applications, Publications de l'institut
	Mathematique, Nouvelle serie, tome 84 (98)  (2008), 61-72.
\end{thebibliography}
\end{document}